\documentclass[a4paper]{article}
\usepackage{amsmath, amsthm, amscd, amsfonts, amssymb, graphicx, color}
\newtheorem{thm}{Theorem}[section]

\newtheorem{prop}[thm]{Proposition}

\newtheorem{lem}[thm]{Lemma}
\newtheorem{Def}[thm]{Definition}
\newtheorem{rem}[thm]{Remark}
\newtheorem{ex}[thm]{Example}

\newcommand{\be}{\begin{equation}}
\newcommand{\ee}{\end{equation}}
\newcommand{\ben}{\begin{enumerate}}
\newcommand{\een}{\end{enumerate}}
\newcommand{\beq}{\begin{eqnarray}}
\newcommand{\eeq}{\end{eqnarray}}
\newcommand{\beqn}{\begin{eqnarray*}}
\newcommand{\eeqn}{\end{eqnarray*}}

\begin{document}
\newcommand{\f}{\frac}
\newtheorem{theorem}{Theorem}[section]
\newtheorem{lemma}[theorem]{Lemma}
\newtheorem{proposition}[theorem]{Proposition}
\newtheorem{corollary}[theorem]{Corollary}
\theoremstyle{definition}
\newtheorem{definition}[theorem]{Definition}
\newtheorem{example}[theorem]{Example}
\newtheorem{solution}[theorem]{Solution}
\newtheorem{xca}[theorem]{Exercise}
\theoremstyle{remark}
\newtheorem{remark}[theorem]{Remark}
\numberwithin{equation}{section}
\newcommand{\sta}{\stackrel}
\title{ On   $(\alpha, \beta, \gamma)$-metrics  }
\author { Nasrin Sadeghzadeh, Tahere Rajabi}
\maketitle
\bigskip
\begin{abstract}
In this paper,  a new class of Finsler metrics which are included $(\alpha, \beta)$-metrics
are introduced. They are defined by a Riemannian metric and two  1-forms $\beta=b_i(x)y^i$ and $\gamma=\gamma_i(x)y^i$. This class of metrics are a generalization of $(\alpha, \beta)$-metrics which are not always $(\alpha, \beta)$-metric. We find a necessary
and sufficient condition for this metric to be locally projectively flat and then
we prove the conditions for this metric to be of  Douglas type.
\end{abstract}

\section{Introduction}
$(\alpha, \beta)$-metrics form a special class of Finsler metrics partly because they are
computable.  An $(\alpha, \beta)$-metric on a smooth manifold $M$ is defined by $F=\alpha\phi(s)$, $s=\frac{\beta}{\alpha}$   where  $\phi=\phi(s)$ is a $C^\infty$ scalar function on $(-b_0, b_0)$ with certain regularity, $\alpha =\sqrt{a_{ij}(x)y^iy^j}$ is a Riemannian metric and  $\beta = b_i(x)y^i$  is a 1-form on $M$.
\\
In \cite{RS}, we had studied   a new generalization  of  $(\alpha, \beta)$-metrics which is defined by  a Finsler metric $F$ and  a 1-form $\gamma= \gamma_iy^i$ on an $n$-dimensional manifold $M$. Then the metric is given by $\bar{F}=F\psi(\tilde{s})$, where $ \tilde{s}:=\frac{\gamma}{F}$,  $\|\gamma\|_F<g_0$ and $\psi(\tilde{s})$ is a positive $C^\infty$ function on $(-g_0,g_0)$.
 One could see these metrics are as a $\beta$-change of a Finsler metric.\\
Suppose that $F=\alpha\phi(s)$, $s=\frac{\beta}{\alpha}$  be an $(\alpha, \beta)$-metric. For any 1-form  $\gamma\neq\beta$, $\bar{F}=\alpha\phi(s)\psi(\tilde{s})$ is not necessarily an $(\alpha, \beta)$-metric.
Notice that if $F=\alpha+\beta$ be a Randers metric and $\bar{F}=F+\gamma$ be a Randers change of $F$, then $\bar{F}=\alpha+\beta+\gamma$ is a Randers metric.
 By this idea we decide to define a new generalization  of  $(\alpha, \beta)$-metrics which is in the form $\bar{F}=\alpha\Psi(s,\bar{s})$, where $\Psi(s,\bar{s})=\phi(s)\psi(\frac{\bar{s}}{\phi(s)})$, $\bar{s}=\frac{\gamma}{\alpha}$.
\\
In this paper, we intend to generalize the  above metric. We consider a new generalization  of  $(\alpha, \beta)$-metrics which is defined by  a Riemannian metric $\alpha =\sqrt{a_{ij}(x)y^iy^j}$ and  two 1-form $\beta= b_iy^i$  and $\gamma= \gamma_iy^i$   on an $n$-dimensional manifold $M$. Then the metric is given by $F=\alpha\Psi(s,\bar{s})$, where $s=\frac{\beta}{\alpha}$, $\bar{s}=\frac{\gamma}{\alpha}$,  $\|\beta\|_{\alpha}<b_0$, $\|\gamma\|_{\alpha}<g_0$  and $\Psi(s,\bar{s})$ is a positive $C^\infty$ function on $(-b_0,b_0)\times(-g_0,g_0)$, is a Finsler metric which we call it  $(\alpha, \beta, \gamma)$-metric.
\\
Another  generalization  of  $(\alpha, \beta)$-metrics  are called general $(\alpha, \beta)$-metrics which first introduced by C. Yu and H. Zhu in \cite{YZ}. By definition, a general $(\alpha, \beta)$-metric $F$ can be expressed in the following form
\[
F = \alpha\phi(b^2,s),
\]
where $ b :=\|\beta\|_{\alpha}$. In the future, we can similarly define general $(\alpha, \beta, \gamma)$-metric that it is given by
\[
F = \alpha\phi(b^2,g^2,s,\bar{s}),
\]
where $ b :=\|\beta\|_{\alpha}$ and $ g:=\|\gamma\|_{\alpha}$.

\section{ Preliminaries }
Let $M$ be a smooth manifold and $TM:=\bigcup_{x\in M} T_xM$ be the tangent bundle of $M$, where $T_xM$ is the tangent space at $x\in M$. A Finsler metric on $M$ is a function $F:TM\longrightarrow [0,+\infty)$ with the following properties
\begin{enumerate}
\item[-] $F$ is $C^\infty$ on $TM\backslash \{0\}$;
\item[-] $F$ is positively 1-homogeneous on the fibers of tangent bundle $TM$;
\item[-] for each  $x\in M$, the following quadratic form $\textbf{g}_y$ on $T_xM$ is positive definite,
\[
\textbf{g}_y(u,v):=\frac{1}{2}\frac{\partial^2}{\partial s\partial t}\big[ F^2(y+su+tv)    \big]|_{t,s=0},    \quad u,v\in T_xM.
\]
\end{enumerate}
Let $x\in M$ and $F_x:= F|_{T_xM}$. To measure the non-Euclidean feature of $F_x$, define $\textbf{C}_y : T_xM\otimes T_xM\otimes T_xM\rightarrow \mathbb{R}$  by
\[
\textbf{C}_y(u,v,w):=\frac{1}{2}\frac{d}{dt}\big[ \textbf{g}_{y+tw}(u,v)    \big]|_{t=0},    \quad u,v,w\in T_xM.
\]
The family $\textbf{C}:=\{\textbf{C}_y\}_{y\in TM_0}$ is called the Cartan torsion. It is well known that
$\textbf{C}=0$ if and only if $F$ is Riemannian.\\
Given a Finsler manifold $(M, F)$, then a global vector field $\textbf{G}$ is induced by $F$ on $TM_0$, which
in a standard coordinate $(x^i, y^i)$ for $TM_0$ is given by
\[
\textbf{G}=y^i\frac{\partial}{\partial x^i}-2G^i(x,y)\frac{\partial}{\partial y^i}.
\]
where $G^i(x, y)$ are local functions on $TM_0$  given by
\be\label{0G^i}
G^i=\frac{1}{4}g^{il}\Big\{     \frac{\partial g_{jl}}{\partial x^k}+\frac{\partial g_{lk}}{\partial x^j}-\frac{\partial g_{jk}}{\partial x^l}     \Big\}y^jy^k.
\ee
$\textbf{G}$ is called the associated spray to $(M, F)$. The projection of an integral curve of the spray
$\textbf{G}$ is called a geodesic in $M$.\\
A Finsler metric $F=F(x,y)$ on an open subset $\mathcal{U} \subseteq \mathbb{R}^n$
is said to be projectively flat if all geodesics are straight in  $\mathcal{U} $.
It is well-known that a Finsler metric $F$ on an open subset $\mathcal{U} \subseteq \mathbb{R}^n$
is projectively flat if and only if it satisfies the following system of equations,
\[
F_{x^ky^j}y^k-F_{x^j}=0.
\]
This fact is due to G. Hamel\cite{H}. In this case, $G^i=Py^i$, where $P=P(x,y)$ is given by
\[
P=\frac{F_{x^k}y^k}{2F}.
\]
The scalar function $P$ is called the projective factor of $F$.

\section{  $(\alpha, \beta, \gamma)$-metrics }

\begin{Def}\label{def1}
For a Riemannian metric $\alpha $ and  two 1-form $\beta=b_i(x)y^i$ and $\gamma=\gamma_i(x)y^i$ on an $n$-dimensional manifold $M$,  an $(\alpha, \beta, \gamma)$-metric $F$  can be expressed as the form
\[
F=\alpha\Psi(s,\bar{s}),\quad s:=\frac{\beta}{\alpha},\quad \bar{s}:=\frac{\gamma}{\alpha},
\]
where  $\|\beta\|_\alpha<b_0$, $\|\gamma\|_\alpha<g_0$ and $\Psi(s,\bar{s})$ is a positive $C^\infty$ function on $(-b_0,b_0)\times(-g_0,g_0)$.
\end{Def}
\begin{prop}\label{prop1}
For an $(\alpha, \beta, \gamma)$-metric $F=\alpha\Psi(s,\bar{s})$, where $s=\frac{\beta}{\alpha}$ and $\bar{s}=\frac{\gamma}{\alpha}$, the fundamental
tensor is given by
\beq\nonumber\label{g_ij}
&&g_{ij}=\rho a_{ij}+\rho_0b_ib_j+\bar{\rho}_0\gamma_i\gamma_j+\rho_1(b_i\alpha_j+b_j\alpha_i)+\bar{\rho}_1(\gamma_i\alpha_j+\gamma_j\alpha_i)
+\rho_2\alpha_i\alpha_j+\rho_3(b_i\gamma_j+b_j\gamma_i),\\
\eeq
where
\beq\nonumber
&&\rho:=\Psi(\Psi-s\Psi_s-\bar{s}\Psi_{\bar{s}}),\quad \  \rho_0:=\Psi\Psi_{ss}+\Psi_s\Psi_s,\qquad \rho_1:=\Psi\Psi_s-s\rho_0-\bar{s}\rho_3, \\\nonumber
&&\rho_2:=-s\rho_1-\bar{s}\bar{\rho}_1,\qquad \qquad    \bar{\rho}_0:=\Psi\Psi_{\bar{s}\bar{s}}+\Psi_{\bar{s}}\Psi_{\bar{s}},\qquad \bar{\rho}_1:=\Psi\Psi_{\bar{s}}-\bar{s}\bar{\rho}_0-s\rho_3,\\    \label{rho}
&&\qquad\qquad\qquad\qquad \qquad\qquad  \rho_3:=\Psi\Psi_{s\bar{s}}+\Psi_s\Psi_{\bar{s}}.
\eeq
Moreover,
\be\label{det}
det(g_{ij})=\Psi^{n+1}\big(\Psi-s\Psi_s-\bar{s}\Psi_{\bar{s}}\big)^{n-2} \Gamma\  det(a_{ij}),
\ee
where
\beq\nonumber\label{Gamma}\nonumber
\Gamma:=&&\!\!\!\!\!\!\!\!\!\!\!\!\Psi-s\Psi_s-\bar{s}\Psi_{\bar{s}}+(b^2-s^2)\Psi_{ss}+(g^2-\bar{s}^2)\Psi_{\bar{s}\bar{s}}+2(\theta-s\bar{s})\Psi_{s\bar{s}}\\
&&\!\!\!\!\!\!\!\!\!\!\!\!+\big[ (b^2-s^2)(g^2-\bar{s}^2)-(\theta-s\bar{s})^2   \big]J,
\eeq
and
\[
b^2:=a^{ij}b_ib_j,  \quad g^2:=a^{ij}\gamma_i\gamma_j, \quad \theta:=a^{ij}b_i\gamma_j,
\]
\[
J:=\frac{\Psi_{ss}\Psi_{\bar{s}\bar{s}}-\Psi_{s\bar{s}}\Psi_{s\bar{s}}}{\Psi-s\Psi_s-\bar{s}\Psi_{\bar{s}}}.
\]
\beq\label{g^ij}\nonumber
g^{ij}=\frac{1}{\rho}\bigg\{\!\!\!\!\!\!\!\!\!\!\!&& a^{ij}-\frac{1}{\Gamma}\Big[\Psi_{ss}+(g^2-\bar{s}^2)J \Big]b^ib^j
-\frac{1}{\Gamma}\Big[\Psi_{\bar{s}\bar{s}}+(b^2-s^2)J \Big]\gamma^i\gamma^j\\\nonumber
&&-\frac{1}{\Psi\Gamma}\Big[\rho_1+\pi_2(\theta-s\bar{s})-\pi_1(g^2-\bar{s}^2) \Big](b^i\alpha^j+b^j\alpha^i)\\ \nonumber
&&-\frac{1}{\Psi\Gamma}\Big[\bar{\rho}_1-\pi_2(b^2-s^2)+\pi_1(\theta-s\bar{s}) \Big](\gamma^i\alpha^j+\gamma^j\alpha^i)\\ \nonumber
&&+\frac{1}{\Psi^2\Gamma}\bigg( \Big[ s\Psi+(b^2-s^2)\Psi_s+(\theta-s\bar{s})\Psi_{\bar{s}} \Big]  \Big[\rho_1+\pi_2(\theta-s\bar{s})-\pi_1(g^2-\bar{s}^2) \Big]\\ \nonumber
&&\qquad\quad+\Big[ \bar{s}\Psi+(g^2-\bar{s}^2)\Psi_{\bar{s}}+(\theta-s\bar{s})\Psi_{s} \Big]  \Big[\bar{\rho}_1-\pi_2(b^2-s^2)+\pi_1(\theta-s\bar{s}) \Big]  \bigg)\alpha^i\alpha^j  \bigg\},\\
\eeq
where
\beq\nonumber
&&\pi_1:=\Psi_{\bar{s}}\Psi_{s\bar{s}}-\Psi_s\Psi_{\bar{s}\bar{s}}+s\Psi J, \\ \label{pi}
&&\pi_2:=\Psi_{s}\Psi_{s\bar{s}}-\Psi_{\bar{s}}\Psi_{ss}+\bar{s}\Psi J.
\eeq
Moreover, the Cartan tensor of $F$ is given by
\beq\nonumber \label{C_ijk}
C_{ijk}=\!\!\!\!\!\!\!\!\!\!\!&&\frac{\rho_1}{2}\Big[ h_k\alpha_{ij}+h_i\alpha_{jk}+h_j\alpha_{ik} \Big]+\frac{\bar{\rho}_1}{2} \Big[\bar{h}_k\alpha_{ij}+\bar{h}_i\alpha_{jk}+\bar{h}_j\alpha_{ik} \Big]\\ \nonumber
&&
+\frac{(\rho_0)_{\bar{s}}}{2\alpha} \Big[ h_ih_j\bar{h}_k+h_jh_k\bar{h}_i+h_ih_k\bar{h}_j \Big]+\frac{(\bar{\rho}_0)_{s}}{2\alpha} \Big[ \bar{h}_i\bar{h}_jh_k+\bar{h}_j\bar{h}_kh_i+\bar{h}_i\bar{h}_kh_j \Big]\\
&&+\frac{(\rho_0)_{s}}{2\alpha}h_ih_jh_k+\frac{(\bar{\rho}_0)_{\bar{s}}}{2\alpha}\bar{h}_i\bar{h}_j\bar{h}_k.
\eeq
\end{prop}
\begin{rem}
One could easily show that the above preposition satisfies for any $(\alpha,\beta)$-metric  just by putting $\bar{s}=0$, and satisfies for any $(\alpha,\gamma)$-metric  just by putting $s=0$.
\end{rem}
\begin{proof}
Recall that the fundamental tensor and Cartan tensor of a Finsler metric $F$ are  given by $g_{ij}=\frac{1}{2}[F^2]_{y^iy^j}=FF_{y^iy^j}+F_{y^i}F_{y^j}$ and $C_{ijk}=\frac{1}{2}(g_{ij})_{y^k}$, respectively.
Direct computations yield
\beqn
&&s_{y^i}=\frac{1}{\alpha}h_i,\ \ where\ \ h_i:=b_i-s\alpha_i,\ \  \alpha_i=\alpha_{y^i},  \\
&&\bar{s}_{y^i}=\frac{1}{\alpha}\bar{h}_i,\ \ where  \ \ \bar{h}_i:=\gamma_i-\bar{s}\alpha_i,\\
&&\Psi_{y^i}=\frac{1}{\alpha}\big[ \Psi_sh_i+\Psi_{\bar{s}}\bar{h}_i \big],\\
&&(\Psi_s)_{y^i}=\frac{1}{\alpha}\big[ \Psi_{ss}h_i+\Psi_{s\bar{s}}\bar{h}_i \big],\\
&&(\Psi_{\bar{s}})_{y^i}=\frac{1}{\alpha}\big[ \Psi_{\bar{s}s}h_i+\Psi_{\bar{s}\bar{s}}\bar{h}_i \big],\\
&&(h_i)_{y^j}=-\frac{1}{\alpha}h_j\alpha_i-s\alpha_{ij}, \ \ where \ \ \alpha_{ij}=\alpha_{y^iy^j}=\frac{1}{\alpha}(a_{ij}-\alpha_i\alpha_{j}), \\
&&(\bar{h}_i)_{y^j}=-\frac{1}{\alpha}\bar{h}_j\alpha_i-\bar{s}\alpha_{ij}.  \\
\eeqn
Let $\ell_i=F_{y^i}$ and $\ell_{ij}=F_{y^iy^j}$. By above equations we have
\be\label{F_i}
\ell_i=\Psi\alpha_i+\Psi_sh_i+\Psi_{\bar{s}}\bar{h}_i,
\ee
\be\label{F_ij}
\ell_{ij}=\big[ \Psi-s\Psi_s-\bar{s}\Psi_{\bar{s}} \big]\alpha_{ij}+\frac{1}{\alpha}\Psi_{ss}h_ih_j+
\frac{1}{\alpha}\Psi_{\bar{s}\bar{s}}\bar{h}_i\bar{h}_j+\frac{1}{\alpha}\Psi_{s\bar{s}}\big[ h_i\bar{h}_j+h_j\bar{h}_i \big],
\ee
 Then we get \eqref{g_ij}.
We can rewrite \eqref{g_ij} as follow
\[
\bar{g}_{ij}=\rho \bigg\{ a_{ij}+\delta_1b_ib_j+\delta_2\gamma_i\gamma_j+\delta_0(b_i+\gamma_i)(b_j+\gamma_j)
+\frac{\rho_2}{\rho}[\alpha_i+\frac{\rho_1}{\rho_2}b_i+\frac{\bar{\rho}_1}{\rho_2}\gamma_i]
[\alpha_j+\frac{\rho_1}{\rho_2}b_j+\frac{\bar{\rho}_1}{\rho_2}\gamma_j] \bigg\}
,
\]
where
\[
\delta_0:=\frac{1}{\rho}(\rho_3-\frac{\rho_1\bar{\rho}_1}{\rho_2}), \qquad \delta_1:=\frac{1}{\rho}(\rho_0-\frac{\rho_1^2}{\rho_2})-\delta_0, \qquad \delta_2:=\frac{1}{\rho}(\bar{\rho}_0-\frac{\bar{\rho}_1^2}{\rho_2})-\delta_0, \qquad
\]
Using Lemma 1.1.1  in \cite{Shen} four times, we obtain \eqref{det} and \eqref{g^ij}.
\end{proof}
\begin{rem}
Notice  that by Cauchy-Schwartz inequality we have
\[
\theta^2=(a^{ij}b_i\gamma_j)^2\leq (a^{ij}b_ib_j)(a^{ij}\gamma_i\gamma_j)=b^2g^2.
\]
\end{rem}
 We need to prove the following proposition.
 \begin{prop}\label{*}
 Let $M$ be an $n$-dimensional manifold. An $(\alpha, \beta, \gamma)$-metric $F=\alpha\Psi(s,\bar{s})$, $s=\frac{\beta}{\alpha}$, $\bar{s}=\frac{\gamma}{\alpha}$  is a
Finsler metric  for any Riemannian $\alpha$ and 1-forms $\beta=b_iy^i, \gamma=\gamma_iy^i$  where $\|\beta\|_\alpha < b_0$, $\|\gamma\|_\alpha < g_0$,  $\theta-s\bar{s}\geq0$ if and only if the positive  $C^\infty$ function $\Psi=\Psi(s,\bar{s})$ satisfying
\be\label{condition0}
\Pi:=\Psi-s\Psi_s-\bar{s}\Psi_{\bar{s}}>0, \qquad \Gamma>0,
\ee
when $n \geq 3$ or
\[
\Gamma>0,
\]
when $n=2$,
where $\Gamma$ is given by \eqref{Gamma} and $s$, $\bar{s}$, $b$, $g$ are arbitrary numbers with $|s|\leq b< b_0$ and $|\bar{s}|\leq g< g_0$.
 \end{prop}
\begin{proof}
The case $n = 2$ is similar to $n\geq3$, so we only prove the proposition for $n\geq3$.
It is easy to verify $F$ is a function with regularity and positive homogeneity. In the following we will verify strong covexity.\\
Assume that (\ref{condition0}) is satisfied,
 then we could write $\Pi\Gamma$ as a  second order equation in $\Pi$ as follows
\be\label{PG}
\Pi\Gamma=\Pi^2+(a+\bar{a})\Pi+(a\bar{a}-b\bar{b})>0,
\ee
where
\beqn
&&a:=(b^2-s^2)\Psi_{ss}+(\theta-s\bar{s})\Psi_{s\bar{s}},\qquad b:=(b^2-s^2)\Psi_{s\bar{s}}+(\theta-s\bar{s})\Psi_{\bar{s}\bar{s}},\\
&&\bar{a}:=(g^2-\bar{s}^2)\Psi_{\bar{s}\bar{s}}+(\theta-s\bar{s})\Psi_{s\bar{s}}, \qquad \bar{b}:=(g^2-\bar{s}^2)\Psi_{s\bar{s}}+(\theta-s\bar{s})\Psi_{ss}.
\eeqn
The above inequality holds iff
\begin{enumerate}
\item[(i)] $\Delta<0$ where $\Delta=(a+\bar{a})^2-4(a\bar{a}-b\bar{b})$;
\item[] or
\item[(ii)] $\Delta=0$, then $\Pi\neq\omega$ and $\Pi\Gamma=(\Pi-\omega)^2$ where $\omega=-\frac{1}{2} (a+\bar{a})$;
\item[] or
\item[(iii)] $\Delta>0$, then $0<\Pi<\omega_1$ or $\Pi>\omega_2$;\\
where
\[\omega_1:=-\frac{1}{2}\big[ (a+\bar{a})+\sqrt{\Delta} \big], \qquad \omega_2:=-\frac{1}{2}\big[ (a+\bar{a})-\sqrt{\Delta} \big].\]
\end{enumerate}
Note that $\omega_1<\omega_2$.
Consider a family of functions as follows
\[ \Psi_t(s,\bar{s})=1-t+t\Psi(s,\bar{s}), \qquad 0\leq t\leq1.\]
Put $F_t=\alpha\Psi_t(s,\bar{s})$ and $g^t_{ij}=\frac{1}{2}[F_t^2]_{y^iy^j}$, then $F_0=\alpha$ and $F_1=F$. We are going to prove $\Pi_t>0$ and $\Gamma_t>0$ for any $0\leq t\leq1$, $|s|\leq b< b_0$ and $|\bar{s}|\leq g< g_0$. It is easy to
see that
\[ \Pi_t=1-t+t\Pi>0.\]
Moreover
\[
\Pi_t\Gamma_t=\Pi_t^2+t(a+\bar{a})\Pi_t+t^2(a\bar{a}-b\bar{b})
\]
Then we have $\Delta_t=t^2\Delta$ where
\be\label{Delta}
\Delta=(a+\bar{a})^2-4(a\bar{a}-b\bar{b}).
\ee
It is easy to see that
for $\Delta_t(s,\bar{s})<0$, the equation  $\Pi_t\Gamma_t$ is always positive, i.e.  $\Gamma_t>0$.\\
Now suppose that
 there are $t_0$ and $(s_0,\bar{s}_0)$ such that $\Delta_{t_0}(s_0,\bar{s}_0)>0$. Since $\Delta_{t}(s,\bar{s})$ is continuous with respect to $t$ and $(s, \bar{s})$, then there is $D\subset (-b_0,b_0)\times(-g_0,g_0)$ such that
    \[
    \forall (s, \bar{s})\in D \qquad \Delta_{t}(s,\bar{s})>0,
    \]\[
    \forall (s, \bar{s})\in \partial D \qquad \Delta_{t}(s,\bar{s})=0,
    \]
    where $\partial D$ is border of $D$. Then on $D$ we have
    \be\label{pos}
    \Pi_t\Gamma_t=(\Pi_t-t\omega_1)(\Pi_t-t\omega_2).
    \ee
    If on $D$ we have $\Gamma_t(s, \bar{s})>0$, then there is not anything to prove. Now suppose that there exits $\mathcal{U}\subset D$ such that for $(s, \bar{s})\in \overline{\mathcal{U}}=\mathcal{U}\bigcup\partial\mathcal{U}$ we have $\Gamma_t(s, \bar{s})\leq 0$. Since $\Gamma_0, \Gamma_1$ are both positive, then by continuity $\Gamma_t$ we get
    \[
    \exists t_1,t_2\in (0, 1)\quad  s.t.\quad \Gamma_{t_1}(s, \bar{s})=\Gamma_{t_2}(s, \bar{s})=0; \quad \forall(s, \bar{s})\in \overline{\mathcal{U}}.
    \]
    By \eqref{pos} we have
    \be\label{t1t2}
    (\Pi_{t_1}-t_1\omega_1)(\Pi_{t_1}-t_1\omega_2)=0, \quad and \quad (\Pi_{t_2}-t_2\omega_1)(\Pi_{t_2}-t_2\omega_2)=0.
    \ee
    Then for $t_1\leq t\leq t_2$ we get
    \[
    \forall(s, \bar{s})\in\overline{\mathcal{U}}\quad \Gamma_t(s, \bar{s})\leq0, \quad and \quad \forall (s, \bar{s})\in D-\overline{\mathcal{U}} \quad \Gamma_t(s, \bar{s})>0.
    \]
    By continuity $\Gamma_t$ we have
    \[
    \Gamma_t(s, \bar{s})=0, \quad t_1\leq t\leq t_2, \quad (s, \bar{s})\in\partial\mathcal{U}.
    \]
    Then  \eqref{pos} yields $\Pi_t=t\omega_1$ or $\Pi_t=t\omega_2$. In this case by \eqref{t1t2} we get $t_1=t_2$ which is a contradiction. So $\Gamma_t(s, \bar{s})>0$ on $D$.\\
Now let there is $D_1\subset (-b_0,b_0)\times(-g_0,g_0)$ such that $\Delta(s,\bar{s})=0$ for every $(s,\bar{s})\in D_1$. Then we see that for every $0\leqslant t\leqslant1$ and $(s,\bar{s}) \in D_1$ we have $\Delta_t(s,\bar{s})=0$. One could easily get
\[
\Pi_t\Gamma_t-t^2\Pi\Gamma=(1-t)\big(1-t+2t(\Pi+\frac{a+\bar{a}}{2})\big).
\]
If for some $0<t<1$ we have $1-t+2t(\Pi+\frac{a+\bar{a}}{2})\geqslant 0$ then $\Pi_t\Gamma_t\geqslant t^2\Pi\Gamma>0$ and therefore $\Gamma_t>0$. Now we assume that there are some $0<t<1$ such that
\be \label{e1e}
1-t+2t(\Pi+\frac{a+\bar{a}}{2})<0.
\ee
which one could easily get
\[
1-t+t(\Pi-\frac{a+\bar{a}}{2})<\frac{1}{2}(1-t)\neq 0.
\]
Thus
\be\label{www}
\Pi_{t}\Gamma_{t}=(\Pi_{t}-\omega_{t})^2=\big(1-t+t(\Pi-\omega)\big)^2=\big(1-t+t(\Pi-\frac{a+\bar{a}}{2})\big)^2>0.
\ee
 Then for this $0<t<1$ we get $\Gamma_t>0$, too.\\

All above arguments yield $\Gamma_t>0$ for any $0\leq t \leq 1$. Then  $det(g^t_{ij}) > 0$ for all $0\leq t \leq 1$. Since $(g^0_{ij} )$ is positive definite, we conclude that $(g^t_{ij})$ is positive definite for any $t \in [0, 1]$. Therefore, $F_t$ is a Finsler metric for any $t \in[0, 1]$.\\
Conversely, assume that $F=\alpha\Psi(s,\bar{s})$ is a Finsler metric for any Riemannian metric $\alpha$ and 1-forms $\beta$ and $\gamma$ with $b < b_0$ and $g<g_0$. Then $\Psi=\Psi(s,\bar{s})$ and $det(g_{ij})$ are positive. By Proposition \ref{prop1},  $det(g_{ij})>0$ is equivalent to
\[ \Pi^{n-2}\Gamma>0 \]
which implies $\Pi\neq0$ when $n \geq 3$. Noting that $\Psi(0,0)>0$, the inequality $\Pi>0$ implies. $\Gamma>0$ also holds because  $det(g_{ij})>0$.
\end{proof}
\begin{ex}
In \cite{RS} had introduced a new class of Finsler metrics that called $(F,\gamma)$-metrics. A Finsler metric $\bar{F}$ is called $(F,\gamma)$-metric if it is in the following form
\[
\bar{F}=F\psi(\tilde{s}),\quad \tilde{s}=\frac{\gamma}{F},
\]
where $F$ is a Finsler metric and $\gamma= \gamma_iy^i$ is a 1-form on an $n$-dimensional manifold $M$, $\psi(\tilde{s})$ is a positive $C^\infty$ function on $(-g_0,g_0)$ and $\|\gamma\|_F<g_0$. It had shown that $\bar{F}$ is a Finsler metric if and only if the positive  $C^\infty$ function $\psi(\tilde{s})$ satisfying
\be\label{cond}
\psi-\tilde{s}\psi'>0, \qquad \psi-\tilde{s}\psi'+(p^2-\tilde{s}^2)\psi''>0,
\ee
when $n \geq 3$ or
\[
\psi-\tilde{s}\psi'+(p^2-\tilde{s}^2)\psi''>0,
\]
when n=2, where $p^2:=g^{ij}\gamma_i\gamma_j$. Now suppose that $F$ is an $(\alpha, \beta)$-metric, i.e. $F=\alpha\phi(s)$, $s=\frac{\beta}{\alpha}$. Then
\be\label{exam}
\bar{F}=\alpha\phi(s)\psi(\tilde{s}).
\ee
Let $\bar{s}=\frac{\gamma}{\alpha}$ and $\Psi:=\phi(s)\psi(\frac{\bar{s}}{\phi(s)})$. Then \eqref{exam} is an $(\alpha, \beta, \gamma)$-metric.
A direct computation gives
\[
\Pi=(\phi-s\phi')(\psi-\tilde{s}\psi'),
\]
\[
\Gamma=\big[\phi-s\phi'+(b^2-s^2)\phi''\big]\big[\psi-\tilde{s}\psi'+(p^2-\tilde{s}^2)\psi''  \big].
\]
By these relations we can conclude that if $F$ be an $(\alpha, \beta)$-metric, then $\bar{F}$ is Finsler metric iff $\Pi>0$ and $\Gamma>0$.
\end{ex}

For 1-form $\beta = b_i(x)y^i$  and  $\gamma=\gamma_i(x)y^i$, we have
\be\label{def r,s}
^\beta\!r_{ij}:=\frac{1}{2}\big(b_{i|j}+b_{j| i}  \big), \qquad ^\beta\! s_{ij}:=\frac{1}{2}\big(b_{i| j}-b_{j|i}  \big).
\ee
\be\label{def r,s2}
^\gamma\! r_{ij}:=\frac{1}{2}\big(\gamma_{i|j}+\gamma_{j| i}  \big), \qquad ^\gamma\! s_{ij}:=\frac{1}{2}\big(\gamma_{i| j}-\gamma_{j|i}  \big).
\ee
where $"|"$   denotes the covariant derivative with respect to the Levi-Civita connection of  $\alpha$.
 Moreover, we define
\[
^\beta\! r_{i0}:=^\beta\!\! r_{ij}y^j, \quad  ^\beta\! r_{j}:= b^i\ ^\beta \!r_{ij}, \quad ^\beta\! r_0:=^\beta\!\! r_jy^j, \quad ^\beta\! r_{00}=^\beta\! \! r_{ij}y^iy^j,
\]
\[
^\beta\! s_{i0}:=^\beta\!\! s_{ij}y^j, \quad  ^\beta s_{j}:=b^i\   ^\beta\! s_{ij}, \quad ^\beta s_0:=^\beta\! s_jy^j, \quad ^\beta\! s^i_0=a^{ij} \ ^\beta\! s_{j0},
\]
\[
^\beta\!\bar{s}_0:=^\beta\!s^i_0\gamma_i.
\]
And
\[
^\gamma r_{i0}:=^\gamma\!\! r_{ij}y^j, \quad  ^\gamma\! r_{j}:= b^i\ ^\gamma \!r_{ij}, \quad ^\gamma r_0:=^\gamma\!\! r_jy^j, \quad ^\gamma r_{00}=^\gamma\! \! r_{ij}y^iy^j,
\]
\[
^\gamma\! s_{i0}:=^\gamma\!\! s_{ij}y^j, \quad  ^\gamma\! s_{j}:=b^i\   ^\gamma\! s_{ij}, \quad ^\gamma s_0:=^\gamma\!\! s_jy^j, \quad ^\gamma s^i_0=a^{ij} \ ^\gamma\! s_{j0},
\]
\[
^\gamma\!\bar{s}_0:=^\gamma\!s^i_0b_i.
\]

\section{Spray coefficients of $F$}
Computing $G^i$ by \eqref{0G^i} is too long. Then we use a different technique  that Matsumoto had used in \cite{M1}.\\
For $F=\alpha\Psi(s,\bar{s})$ we can get
\beq\nonumber\label{sx}
&&\beta_{x^j}=b_{0|j}+b_rG^r_j, \qquad  \qquad \gamma_{x^j}=\gamma_{0|j}+\gamma_rG^r_j, \\
&&s_{x^j}=\frac{1}{\alpha}(b_{0|j}+h_rG^r_j),  \qquad
\bar{s}_{x^j}=\frac{1}{\alpha}(\gamma_{0|j}+\bar{h}_rG^r_j),
\eeq
where $G^i_j=^\alpha\!\!G^i_{y^j}$. Moreover, by $\alpha_{|i}=0$ and $\alpha_{i|j}=0$ we have
\beq\nonumber\label{alpha_x}
&&\alpha_{x^j}=\alpha_rG^r_j, \\
&&(\alpha_i)_{x^j}=\alpha_{ir}G^r_j+\alpha_rG^r_{ij},
\eeq
where $G^r_{ij}=^\alpha\!\!G^r_{y^iy^j}$. Then
\beq\label{hx}\nonumber
&&(h_i)_{x^j}=b_{i|j}-\frac{1}{\alpha}b_{0|j}\alpha_i-\frac{1}{\alpha}h_rG^r_j\alpha_i+h_rG^r_{ij}-s\alpha_{ir}G^r_j,\\
&&(\bar{h}_i)_{x^j}=\gamma_{i|j}-\frac{1}{\alpha}\gamma_{0|j}\alpha_i-\frac{1}{\alpha}\bar{h}_rG^r_j\alpha_i+\bar{h}_rG^r_{ij}-\bar{s}\alpha_{ir}G^r_j.
\eeq
Differentiating  (\ref{F_i}) with respect to $x^j$ and using  (\ref{sx}), (\ref{alpha_x}) and \eqref{hx} yield
\beq\nonumber\label{lx2}
\frac{\partial \ell_i}{\partial x^j}=\!\!\!\!\!\!\!\!\!\!\!&&\Psi_s b_{i|j}+\Psi_{\bar{s}} \gamma_{i|j}+\frac{1}{\alpha}\Big[\Psi_{ss}b_{0|j}+\Psi_{s\bar{s}} \gamma_{0|j} \Big]h_i+\frac{1}{\alpha}\Big[\Psi_{s\bar{s}}b_{0|j}+\Psi_{\bar{s}\bar{s}} \gamma_{0|j} \Big]\bar{h}_i\\ \nonumber
&&+\Big[\Psi\alpha_r+\Psi_{s}h_r+\Psi_{\bar{s}} \bar{h}_r \Big]G^r_{ij}+(\Psi-s\Psi_{s}-\bar{s}\Psi_{\bar{s}})\alpha_{ir} G^r_{j}\\
&&+\frac{1}{\alpha}\Big[ \Psi_{ss}h_ih_r+\Psi_{\bar{s}\bar{s}}\bar{h}_i\bar{h}_j+\Psi_{s\bar{s}}(h_i\bar{h}_j+\bar{h}_ih_j)  \Big]G^r_j.
\eeq
Let $";"$  denotes the horizontal   covariant derivative with respect to Cartan connection of $F$.
Next, we deal with $\ell_{i;j}=0$, that is $\frac{\partial\ell_i}{\partial x^j}=\ell_{ir}N^r_j+\ell_r\Gamma^r_{ij}$. Let us define
\be\label{D}
D^i_{jk}:=\Gamma^i_{jk}-G^i_{jk}, \qquad D^i_{j}:=D^i_{jk}y^k=N^i_{j}-G^i_{j}, \qquad D^i:=D^i_{j}y^j=2G^i-2\ ^\alpha\!G^i.
\ee
Then
\[
\frac{\partial\ell_i}{\partial x^j}=\ell_{ir}(D^r_j+G^r_j)+\ell_r(D^r_{ij}+G^r_{ij}).
\]
Putting  (\ref{F_i}) and (\ref{F_ij}) in above equation yields
\beq\nonumber\label{lx3}
\frac{\partial\ell_i}{\partial x^j}=\!\!\!\!\!\!\!\!\!\!\!&&\ell_{ir}D^r_j+\ell_rD^r_{ij}+\Big[\Psi\alpha_r+\Psi_{s}h_r+\Psi_{\bar{s}} \bar{h}_r \Big]G^r_{ij}\\\label{lx3}
&&+\Big[ (\Psi-s\Psi_s-\bar{s}\Psi_{\bar{s}})\alpha_{ir}+\frac{1}{\alpha}\Psi_{ss}h_ih_r+\frac{1}{\alpha}\Psi_{\bar{s}\bar{s}}\bar{h}_i\bar{h}_r
+\frac{1}{\alpha}\Psi_{s\bar{s}}(h_i\bar{h}_r+\bar{h}_rh_i)\Big]G^r_j.
\eeq
By comparing (\ref{lx2}) and (\ref{lx3}) we get the following
\be\label{b_i|j}
\Psi_sb_{i|j}+\Psi_{\bar{s}}\gamma_{i|j}=\ell_{ir}D^r_j+\ell_{r}D^r_{ij}-\frac{1}{\alpha}\big[ \Psi_{ss}b_{0|j}+\Psi_{s\bar{s}}\gamma_{0|j} \big]h_i-\frac{1}{\alpha}\big[ \Psi_{s\bar{s}}b_{0|j}+\Psi_{\bar{s}\bar{s}}\gamma_{0|j} \big]\bar{h}_i.
\ee
Thus by (\ref{def r,s}) and \eqref{def r,s2} we have
\beq\nonumber\label{r_ij}
2\Psi_s\ ^\beta\!r_{ij}+2\Psi_{\bar{s}}\ ^\gamma\!r_{ij}\!\!\!\!\!\!\!\!\!\!\!&&=\ell_{ir}D^r_j+\ell_{jr}D^r_i+2\ell_{r}D^r_{ij}\\ \nonumber
&&-\frac{1}{\alpha}\big[ \Psi_{ss}b_{0|j}+\Psi_{s\bar{s}}\gamma_{0|j} \big]h_i-\frac{1}{\alpha}\big[ \Psi_{ss}b_{0|i}+\Psi_{s\bar{s}}\gamma_{0|i} \big]h_j\\
&&-\frac{1}{\alpha}\big[ \Psi_{s\bar{s}}b_{0|j}+\Psi_{\bar{s}\bar{s}}\gamma_{0|j} \big]\bar{h}_i-\frac{1}{\alpha}\big[ \Psi_{s\bar{s}}b_{0|i}+\Psi_{\bar{s}\bar{s}}\gamma_{0|i} \big]\bar{h}_j,  \\ \nonumber
2\Psi_s\ ^\beta\!s_{ij}+2\Psi_{\bar{s}}\ ^\gamma\!s_{ij}\!\!\!\!\!\!\!\!\!\!\!&&=\ell_{ir}D^r_j-\ell_{jr}D^r_i\\
\nonumber
&&-\frac{1}{\alpha}\big[ \Psi_{ss}b_{0|j}+\Psi_{s\bar{s}}\gamma_{0|j} \big]h_i+\frac{1}{\alpha}\big[ \Psi_{ss}b_{0|i}+\Psi_{s\bar{s}}\gamma_{0|i} \big]h_j\\\label{s_ij}
&&-\frac{1}{\alpha}\big[ \Psi_{s\bar{s}}b_{0|j}+\Psi_{\bar{s}\bar{s}}\gamma_{0|j} \big]\bar{h}_i+\frac{1}{\alpha}\big[ \Psi_{s\bar{s}}b_{0|i}+\Psi_{\bar{s}\bar{s}}\gamma_{0|i} \big]\bar{h}_j.
\eeq
Contracting (\ref{r_ij}) and (\ref{s_ij}) with $y^j$ implies that
\beq\nonumber\label{r_i0}
2\Psi_s\ ^\beta\!r_{i0}+2\Psi_{\bar{s}}\ ^\gamma\!r_{i0}\!\!\!\!\!\!\!\!\!\!\!&&=\ell_{ir}D^r+2\ell_{r}D^r_{i}-\frac{1}{\alpha}\big[ \Psi_{ss}\ ^\beta\!r_{00}+\Psi_{s\bar{s}}\ ^\gamma\!r_{00} \big]h_i\\
&&-\frac{1}{\alpha}\big[ \Psi_{s\bar{s}}\ ^\beta\!r_{00}+\Psi_{\bar{s}\bar{s}}\ ^\gamma\!r_{00} \big]\bar{h}_i.
\\\nonumber
2\Psi_s\ ^\beta\!s_{i0}+2\Psi_{\bar{s}}\ ^\gamma\!s_{i0}\!\!\!\!\!\!\!\!\!\!\!&&=\ell_{ir}D^r-\frac{1}{\alpha}\big[ \Psi_{ss}\ ^\beta\!r_{00}+\Psi_{s\bar{s}}\ ^\gamma\!r_{00} \big]h_i\\ \label{s_i0}
&&-\frac{1}{\alpha}\big[ \Psi_{s\bar{s}}\ ^\beta\!r_{00}+\Psi_{\bar{s}\bar{s}}\ ^\gamma\!r_{00} \big]\bar{h}_i.
\eeq
Subtracting (\ref{s_i0}) from (\ref{r_i0}) yields
\be\label{r-s}
\Psi_s(\ ^\beta\!r_{i0}-\ ^\beta\!s_{i0})+\Psi_{\bar{s}}(\ ^\gamma\!r_{i0}-\ ^\gamma\!s_{i0})=\ell_rD^r_i.
\ee
Contracting (\ref{r-s})  with $y^i$ leads to
\be\label{r_00}
\Psi_s\ ^\beta\!r_{00}+\Psi_{\bar{s}}\ ^\gamma\!r_{00}=\ell_rD^r.
\ee
To obtain the spray coefficients of $F$, first we propose the following lemma.
\begin{lem}\label{system}
The system of algebraic  equations
\[
(i)\ \ell_{ir}A^r=B_i,   \qquad (ii)\ \ell_{r}A^r=B,
\]
has unique solution $A^r$ for given $B$ and $B_i$ such that $B_iy^i=0$. The solution is given by
\be\label{solution}
A^i=(\alpha_rA^r)\alpha^i+\frac{\alpha}{\Pi}B^i-\frac{\alpha}{\Pi\Gamma}(\mu_1h^i+\mu_2\bar{h}^i),
\ee
where $B^i=a^{il}B_l$, $h^i=a^{il}h_l$, $\bar{h}^i=a^{il}\bar{h}_l$ and
\beqn
&&\Pi:=\Psi-s\Psi_s-\bar{s}\Psi_{\bar{s}}, \\
&&\mu_1:=\big[ \Psi_{ss}+(g^2-\bar{s}^2)J \big]B_rb^r+\big[ \Psi_{s\bar{s}}-(\theta-s\bar{s})J \big]B_r\gamma^r,\\
&&\mu_2:=\big[ \Psi_{\bar{s}\bar{s}}+(b^2-s^2)J \big]B_r\gamma^r+\big[ \Psi_{s\bar{s}}-(\theta-s\bar{s})J \big]B_rb^r.
\eeqn
\end{lem}
\begin{proof}
By contracting (\ref{F_ij}) with $b^i$ and $\gamma^i$ we have
\beq\label{lb}\nonumber
&&\ell_{ij}b^i=\frac{1}{\alpha}\big[\Pi+(b^2-s^2)\Psi_{ss}+(\theta-s\bar{s})\Psi_{s\bar{s}}  \big]h_j
+\frac{1}{\alpha}\big[ (b^2-s^2)\Psi_{s\bar{s}}+(\theta-s\bar{s})\Psi_{\bar{s}\bar{s}} \big]\bar{h}_j,\\
&&\label{lg}\\ \nonumber
&&\ell_{ij}\gamma^i=\frac{1}{\alpha}\big[ (\theta-s\bar{s})\Psi_{ss}+(g^2-\bar{s}^2)\Psi_{s\bar{s}} \big]h_j
+\frac{1}{\alpha}\big[ \Pi+(\theta-s\bar{s})\Psi_{s\bar{s}}+(g^2-\bar{s}^2)\Psi_{\bar{s}\bar{s}}  \big]\bar{h}_j.\\
\eeq
Next contracting equation (i) with $b^i$ and $\gamma^i$ and using \eqref{lb} and \eqref{lg} we get the following
\[
\begin{cases}
&\big[ \Pi+(b^2-s^2)\Psi_{ss}+(\theta-s\bar{s})\Psi_{s\bar{s}}  \big]h_jA^j
+\big[ (b^2-s^2)\Psi_{s\bar{s}}+(\theta-s\bar{s})\Psi_{\bar{s}\bar{s}} \big]\bar{h}_jA^j= \alpha B_jb^j \\ \\
&\big[ (\theta-s\bar{s})\Psi_{ss}+(g^2-\bar{s}^2)\Psi_{s\bar{s}} \big]h_jA^j
+\big[ \Pi+(\theta-s\bar{s})\Psi_{s\bar{s}}+(g^2-\bar{s}^2)\Psi_{\bar{s}\bar{s}}  \big]\bar{h}_jA^j=\alpha B_j\gamma^j.
\end{cases}
\]
By solving the above system we obtain
\beq\nonumber
&&h_jA^j=\frac{\alpha}{\Pi\Gamma}\Big\{  \big[ \Pi+(\theta-s\bar{s})\Psi_{s\bar{s}}+(g^2-\bar{s}^2)\Psi_{\bar{s}\bar{s}}  \big]B_jb^j
-\big[ (b^2-s^2)\Psi_{s\bar{s}}+(\theta-s\bar{s})\Psi_{\bar{s}\bar{s}} \big]B_j\gamma^j         \Big\},\\\label{hA} \\ \nonumber
&&\bar{h}_jA^j=\frac{\alpha}{\Pi\Gamma}\Big\{ \big[ \Pi+(b^2-s^2)\Psi_{ss}+(\theta-s\bar{s})\Psi_{s\bar{s}}  \big]B_j\gamma^j-
\big[ (\theta-s\bar{s})\Psi_{ss}+(g^2-\bar{s}^2)\Psi_{s\bar{s}} \big]B_jb^j \Big\}.\\\label{hA2}
\eeq
 Substituting  (\ref{F_i}) in  equation (ii) yields
\[
\Psi\alpha_jA^j+\Psi_sh_jA^j+\Psi_{\bar{s}}\bar{h}_jA^j=B.
\]
By \eqref{hA} and \eqref{hA2}  we get
\beq\nonumber\label{aA}
\alpha_jA^j=\frac{1}{\Psi}\Big\{ B\!\!\!\!\!\!\!\!\!\!\!&&-\frac{\alpha}{\Pi\Gamma} \Big( \Psi_s\big[ \Pi+(\theta-s\bar{s})\Psi_{s\bar{s}}+(g^2-\bar{s}^2)\Psi_{\bar{s}\bar{s}}  \big]-
\Psi_{\bar{s}}\big[ (\theta-s\bar{s})\Psi_{ss}+(g^2-\bar{s}^2)\Psi_{s\bar{s}} \big]\Big) B_jb^j\\ \nonumber
&&-\frac{\alpha}{\Pi\Gamma}\Big( \Psi_{\bar{s}}\big[ \Pi+(b^2-s^2)\Psi_{ss}+(\theta-s\bar{s})\Psi_{s\bar{s}}  \big]-
\Psi_s\big[ (b^2-s^2)\Psi_{s\bar{s}}+(\theta-s\bar{s})\Psi_{\bar{s}\bar{s}} \big]  \Big)B_j\gamma^j \Big\}.\\
\eeq
Applying    (\ref{F_ij}) in  equation (i)   we obtain
\[
\frac{\Pi}{\alpha}\big[a_{ij}A^j-(\alpha_jA^j)\alpha_i\big]+\frac{1}{\alpha}\big[ (\Psi_{ss}h_i+\Psi_{s\bar{s}}\bar{h}_i)h_jA^j+
(\Psi_{s\bar{s}}h_i+\Psi_{\bar{s}\bar{s}}\bar{h}_i)\bar{h}_jA^j \big]=B_i.
\]
Contracting this equation  with $a^{ij}$ and using \eqref{hA} and \eqref{hA2} one could conclude \eqref{solution}.
\end{proof}
Now, we are able to obtain the spray coefficients of $F$.\\
The equations  (\ref{s_i0}) and (\ref{r_00})  constitute the system of algebraic equations  whose
solution from lemma \ref{system} is given by
\[
D^i=(\alpha_rD^r)\alpha^i+\frac{\alpha}{\Pi}B^i-\frac{\alpha}{\Pi\Gamma}(\mu_1h^i+\mu_2\bar{h}^i),,
\]
where
\beqn
&&B_i=2\Psi_s\ ^\beta\!s_{i0}+2\Psi_{\bar{s}}\ ^\gamma\!s_{i0}+\frac{1}{\alpha}\big[ \Psi_{ss}\ ^\beta\!r_{00}+\Psi_{s\bar{s}}\ ^\gamma\!r_{00} \big]h_i+\frac{1}{\alpha}\big[ \Psi_{s\bar{s}}\ ^\beta\!r_{00}+\Psi_{\bar{s}\bar{s}}\ ^\gamma\!r_{00} \big]\bar{h}_i,\\
&&
B=\Psi_s\ ^\beta\!r_{00}+\Psi_{\bar{s}}\ ^\gamma\!r_{00},\\
&&B_ib^i=2\Psi_s\ ^\beta\!s_{0}+2\Psi_{\bar{s}}\ ^\gamma\!\bar{s}_{0}+\frac{1}{\alpha}\big[ \Psi_{ss}\ ^\beta\!r_{00}+\Psi_{s\bar{s}}\ ^\gamma\!r_{00} \big](b^2-s^2)+\frac{1}{\alpha}\big[ \Psi_{s\bar{s}}\ ^\beta\!r_{00}+\Psi_{\bar{s}\bar{s}}\ ^\gamma\!r_{00} \big](\theta-s\bar{s}),\\
&&B_i\gamma^i=2\Psi_s\ ^\beta\!\bar{s}_{0}+2\Psi_{\bar{s}}\ ^\gamma\!s_{0}+\frac{1}{\alpha}\big[ \Psi_{ss}\ ^\beta\!r_{00}+\Psi_{s\bar{s}}\ ^\gamma\!r_{00} \big](\theta-s\bar{s})+\frac{1}{\alpha}\big[ \Psi_{s\bar{s}}\ ^\beta\!r_{00}+\Psi_{\bar{s}\bar{s}}\ ^\gamma\!r_{00} \big](g^2-\bar{s}^2).
\eeqn
Putting $D^i=2\bar{G}^i-2G^i$ we get the following,
\begin{prop}
The spray coefficients $G^i$ are related to $^\alpha\!G^i$ by
\be\label{G^i}
G^{i}=^{\alpha}\!G^i+\frac{\alpha}{A}\big[ \Psi_s\ ^\beta\!s^i_0+\Psi_{\bar{s}}\ ^\gamma\!s^i_0 \big]+\frac{1}{2\Gamma}\big[\Gamma_1b^i+\Gamma_2\gamma^i+\frac{1}{\Psi}\Gamma_3\alpha^i  \big],
\ee
where
\beq\nonumber\label{gamma123}
&&\Gamma_1:=\big[\Psi_{ss}+(g^2-\bar{s}^2)J \big]\mathcal{R}^\beta+\big[\Psi_{s\bar{s}}-(\theta-s\bar{s})J \big]\ \mathcal{R}^\gamma,\\ \nonumber
&&\Gamma_2:=\big[\Psi_{\bar{s}\bar{s}}+(b^2-s^2)J \big]\mathcal{R}^\gamma+\big[\Psi_{s\bar{s}}-(\theta-s\bar{s})J \big]\mathcal{R}^\beta,\\
&&\Gamma_3:=\big[\rho_1+\pi_2(\theta-s\bar{s})-\pi_1(g^2-\bar{s}^2) \big]\mathcal{R}^\beta+
\big[\bar{\rho}_1-\pi_2(b^2-s^2)+\pi_1(\theta-s\bar{s}) \big]\mathcal{R}^\gamma,
\eeq
and
\beq\nonumber \label{R}
&&\mathcal{R}^\beta:=^\beta\!\!r_{00}-\frac{2\alpha}{\Pi}\big[ \Psi_s\ ^\beta\!s_0+\Psi_{\bar{s}}\ ^\gamma\!\bar{s}_0 \big], \\ \label{R,E}
&&\mathcal{R}^\gamma:=^\gamma\!\!r_{00}-\frac{2\alpha}{\Pi}\big[ \Psi_s \ ^\beta\!\bar{s}_0+\Psi_{\bar{s}}\ ^\gamma\!s_0 \big].
\eeq
\end{prop}

\section{Projectively flat  $(\alpha, \beta, \gamma)$-metrics}
\begin{lem}
An  $(\alpha, \beta, \gamma)$-metric $F=\alpha\Psi(s,\bar{s})$, where $s=\frac{\beta}{\alpha}$  and $\bar{s}=\frac{\gamma}{\alpha}$, is projectivly flat on an open subset $\mathcal{U} \subseteq \mathbb{R}^n$ if and only if
\be\label{condition}
^\alpha\!h_{ij}\ ^\alpha\!G^i+\frac{\alpha}{\Pi}\big[ \Psi_s\ ^\beta\!s_{j0}+\Psi_{\bar{s}}\ ^\gamma\!s_{j0}\big]+\frac{1}{2\Gamma}\big[\Gamma_1h_j+\Gamma_2\bar{h}_j  \big]=0,
\ee
where $\Gamma_1$ and $\Gamma_2$ are given by \eqref{gamma123} and
\[
^\alpha\!h_{ij}=a_{ij}-\alpha_i\alpha_j.
\]
\end{lem}
\begin{proof}
Let  $F=\alpha\Psi(s,\bar{s})$ be a projectively flat metric on $\mathcal{U}$. Therefore, we have
\be\label{P}
G^i=Py^i
\ee
Contracting \eqref{P} with $^\alpha\!h_{ij}$ and using \eqref{G^i}  we get \eqref{condition}.\\
Conversely, suppose that condition \eqref{condition} holds. Contracting \eqref{condition} with $a^{ij}$ yields
\[
\frac{\alpha}{\Pi}\big[ \Psi_s\ ^\beta\!s_{0}^j+\Psi_{\bar{s}}\ ^\gamma\!s_{0}^j\big]=-\frac{1}{2\Gamma}\big[\Gamma_1h^j+\Gamma_2\bar{h}^j  \big]-\big[ ^\alpha\!G^i-^\alpha\!G^r\alpha_r\alpha^i  \big].
\]
Applying it to \eqref{G^i} leads to
\[
G^i=\Big\{\  ^\alpha\!G^r\alpha_r+\frac{1}{2\Gamma}\big[s\Gamma_1+\bar{s}\Gamma_2+\frac{1}{\Psi}\Gamma_3  \big]  \Big\}\alpha^i.
\]
This implies that $F$ is projectively flat.
\end{proof}
\begin{ex}\label{Ex1}
We consider an $(\alpha, \beta, \gamma)$-metric in the following form
\[
F=\alpha e^{\frac{\beta}{\alpha}}+\gamma,\qquad \Psi(s,\bar{s})=e^s+\bar{s}
\]
Let $b_0> 0$ and $g_0>0$ be the largest numbers such that
\be\label{AG}
\Pi=(1-s)e^s>0, \qquad \Gamma=(1-s+b^2-s^2)e^s>0, \qquad |s|<b<b_0, \quad |\bar{s}|<g<g_0,
\ee
so that $F$ is a Finsler metric if and only if $\beta$ and $\gamma$  satisfy  that $b:=\|\beta\|_\alpha<b_0$ and $g:=\|\gamma\|_\alpha<g_0$.
\end{ex}
For this metric we can prove the following lemma.
\begin{lem}
The $(\alpha, \beta, \gamma)$-metric $F =\alpha e^{\frac{\beta}{\alpha}}+\gamma$  is locally projectively flat if and only if $\beta$ is parallel with respect to $\alpha$ and $\gamma$ is closed.
\end{lem}
Recall that 1-form $\beta$ is closed $(d\beta=0)$ if and only if $ ^\beta s_{ij}=0$, and $\beta$ is parallel with respect to $\alpha$ if and only if
$b_{i|j}=0$, i.e. $ ^\beta s_{ij}=0$ and $ ^\beta r_{ij}=0$.
\begin{proof}
let $F =\alpha e^{\frac{\beta}{\alpha}}+\gamma$  be locally projectively flat.  Putting \eqref{AG} into \eqref{condition} yields
\beqn
h_{ij}\ ^\alpha\!G^i\!\!\!\!\!\!\!\!&&+\frac{\alpha^2}{(\alpha-\beta)e^{\frac{\beta}{\alpha}}}\big[e^{\frac{\beta}{\alpha}}\ ^\beta\!s_{j0}+^\gamma\!s_{j0}\big]\\
&&+\frac{\alpha^2}{2\big[ \alpha^2-\alpha\beta+b^2\alpha^2-\beta^2 \big]}\Big\{\ ^\beta\!r_{00}-\frac{2\alpha^2}{(\alpha-\beta)e^{\frac{\beta}{\alpha}}}\big[e^{\frac{\beta}{\alpha}}\ ^\beta\!s_{0}+^\gamma\!\bar{s}_{0}\big]  \Big\}h_j=0.
\eeqn
By multiplying this equation by $2\alpha^2(\alpha-\beta)\big[ \alpha^2-\alpha\beta+b^2\alpha^2-\beta^2 \big]e^{\frac{\beta}{\alpha}}$, we get
\beqn
&&2(\alpha-\beta)\big[ \alpha^2-\alpha\beta+b^2\alpha^2-\beta^2 \big]e^{\frac{\beta}{\alpha}}(a_{ij}\alpha^2-y_iy_j)\ ^\alpha\!G^i\\
&&+
2\alpha^4\big[ \alpha^2-\alpha\beta+b^2\alpha^2-\beta^2 \big]\big[e^{\frac{\beta}{\alpha}}\ ^\beta\!s_{j0}+^\gamma\!s_{j0}\big]\\
&&+\alpha^2(\alpha-\beta)e^{\frac{\beta}{\alpha}}\ ^\beta\!r_{00}(\alpha^2b_j-\beta y_j)-2\alpha^4 \big[e^{\frac{\beta}{\alpha}}\ ^\beta\!s_{0}+^\gamma\!\bar{s}_{0}\big](\alpha^2b_j-\beta y_j)=0.
\eeqn
We can  rewrite this equation as a polynomial in $y^i$ and $\alpha$. This gives
\beqn
&&\Big\{ -2\beta\big[ 2\alpha^2+b^2\alpha^2-\beta^2 \big]e^{\frac{\beta}{\alpha}}(a_{ij}\alpha^2-y_iy_j)\ ^\alpha\!G^i+
2\alpha^4\big[ \alpha^2+b^2\alpha^2-\beta^2 \big]\big[e^{\frac{\beta}{\alpha}}\ ^\beta\!s_{j0}+^\gamma\!s_{j0}\big]\\
&&\quad -\alpha^2\beta e^{\frac{\beta}{\alpha}}\ ^\beta\!r_{00}(\alpha^2b_j-\beta y_j)-2\alpha^4 \big[e^{\frac{\beta}{\alpha}}\ ^\beta\!s_{0}+^\gamma\!\bar{s}_{0}\big](\alpha^2b_j-\beta y_j) \Big\}\\
&&+\alpha\Big\{ 2\big[ \alpha^2+b^2\alpha^2\big]e^{\frac{\beta}{\alpha}}(a_{ij}\alpha^2-y_iy_j)\ ^\alpha\!G^i-2\beta\alpha^4\big[e^{\frac{\beta}{\alpha}}\ ^\beta\!s_{j0}+^\gamma\!s_{j0}\big]+\alpha^2e^{\frac{\beta}{\alpha}}\ ^\beta\!r_{00}(\alpha^2b_j-\beta y_j) \Big\}=0.
\eeqn
$\alpha^{even}$ is rational in $y^i$ and $\alpha$ is
irrational. Then  we have two following equations:
\beq\label{rat1}\nonumber
&&-2\beta\big[ 2\alpha^2+b^2\alpha^2-\beta^2 \big]e^{\frac{\beta}{\alpha}}(a_{ij}\alpha^2-y_iy_j)\ ^\alpha\!G^i+2
\alpha^4\big[ \alpha^2+b^2\alpha^2-\beta^2 \big]\big[e^{\frac{\beta}{\alpha}}\ ^\beta\!s_{j0}+^\gamma\!s_{j0}\big]\\
&&-\alpha^2\beta e^{\frac{\beta}{\alpha}}\ ^\beta\!r_{00}(\alpha^2b_j-\beta y_j)-2\alpha^4 \big[e^{\frac{\beta}{\alpha}}\ ^\beta\!s_{0}+^\gamma\!\bar{s}_{0}\big](\alpha^2b_j-\beta y_j)=0,
\eeq
And
\be\label{iirat1}
2\big[ \alpha^2+b^2\alpha^2\big]e^{\frac{\beta}{\alpha}}(a_{ij}\alpha^2-y_iy_j)\ ^\alpha\!G^i-2\beta\alpha^4\big[e^{\frac{\beta}{\alpha}}\ ^\beta\!s_{j0}+^\gamma\!s_{j0}\big]+\alpha^2e^{\frac{\beta}{\alpha}}\ ^\beta\!r_{00}(\alpha^2b_j-\beta y_j)=0.
\ee
Then we have
\beqn
&&(\alpha^2+b^2\alpha^2)\Big\{ 2 \alpha^4\big[ \alpha^2+b^2\alpha^2-\beta^2 \big]\big[e^{\frac{\beta}{\alpha}}\ ^\beta\!s_{j0}+^\gamma\!s_{j0}\big]\\
&&-\alpha^2\beta e^{\frac{\beta}{\alpha}}\ ^\beta\!r_{00}(\alpha^2b_j-\beta y_j)-2\alpha^4 \big[e^{\frac{\beta}{\alpha}}\ ^\beta\!s_{0}+^\gamma\!\bar{s}_{0}\big](\alpha^2b_j-\beta y_j)  \Big\}\\
&&=-\beta\big[ 2\alpha^2+b^2\alpha^2-\beta^2 \big]\Big\{ -2\beta\alpha^4\big[e^{\frac{\beta}{\alpha}}\ ^\beta\!s_{j0}+^\gamma\!s_{j0}\big]+\alpha^2e^{\frac{\beta}{\alpha}}\ ^\beta\!r_{00}(\alpha^2b_j-\beta y_j)  \Big\}.
\eeqn
Therefore
\beq\label{P0}\nonumber
&&2\alpha^2\Big\{ (\alpha^2+b^2\alpha^2-\beta^2)^2-\alpha^2\beta^2 \Big\}\big[e^{\frac{\beta}{\alpha}}\ ^\beta\!s_{j0}+^\gamma\!s_{j0}\big]\\
&&+
\Big\{ \beta(\alpha^2-\beta^2)e^{\frac{\beta}{\alpha}}\ ^\beta\!r_{00}-2\alpha^2(\alpha^2+b^2\alpha^2)\big[e^{\frac{\beta}{\alpha}}\ ^\beta\!s_{0}+^\gamma\!\bar{s}_{0}\big] \Big\}(\alpha^2b_j-\beta y_j) =0.
\eeq
Contracting \eqref{P0} with $b^j$ leads to
\be\label{P1}
2\alpha^2(\alpha^2-\beta^2)(\alpha^2+b^2\alpha^2-\beta^2)(e^{\frac{\beta}{\alpha}}\ ^\beta\!s_0+^\gamma\!\bar{s}_0)+\beta(\alpha^2-\beta^2)(b^2\alpha^2-\beta^2)e^{\frac{\beta}{\alpha}}\ ^\beta\!r_{00}=0.
\ee
Since $\alpha^2\not\equiv 0 \pmod \beta$ Then $\alpha^2-\beta^2\ne 0$. The term of \eqref{P1} which does not contain $\alpha^2$ is $-\beta^3e^{\frac{\beta}{\alpha}}\ ^\beta\!r_{00}$. Notice $-\beta^3e^{\frac{\beta}{\alpha}}$ is not divisible by $\alpha^2$, then $^\beta\!r_{00}=k(x)\alpha^2$ where we can  consider two cases:
\ben
\item [(1)]  $k(x)=0$;
\item [(2)]  $k(x)\neq0$.
\een
{\bf Case 1:} Substituting $^\beta\!r_{00}=0$ into \eqref{P1} implies that
\[
(\alpha^2+b^2\alpha^2-\beta^2)(e^{\frac{\beta}{\alpha}}\ ^\beta\!s_0+^\gamma\!\bar{s}_0)=0.
\]
If $\alpha^2+b^2\alpha^2-\beta^2=0$, then the term which does not contain $\alpha^2$  is $\beta^2$, which implies that $\beta^2=0$ and is a contradiction. Hence
\be\label{P2}
e^{\frac{\beta}{\alpha}}\ ^\beta\!s_0+^\gamma\!\bar{s}_0=0.
\ee
Putting $^\beta\!r_{00}=0$ and \eqref{P2} into \eqref{P0} leads to
\[
\Big[ (\alpha^2+b^2\alpha^2-\beta^2)^2-\alpha^2\beta^2 \Big]\big[e^{\frac{\beta}{\alpha}}\ ^\beta\!s_{j0}+^\gamma\!s_{j0}\big]=0.
\]
If $(\alpha^2+b^2\alpha^2-\beta^2)^2-\alpha^2\beta^2 =0$, then by a similar argument, we get $\beta^4=0$ which is a contradiction. Therefore
\be\label{P3}
e^{\frac{\beta}{\alpha}}\ ^\beta\!s_{i0}+^\gamma\!s_{i0}=0.
\ee
Differentiating \eqref{P3} with respect to  $y^j$ and  $y^k$  imply that
\[
-(\alpha_jh_k+\alpha_kh_j-s\alpha\alpha_{jk})\ ^\beta\!s_{i0}+h_jh_k\ ^\beta\!s_{i0}+\alpha h_j\ ^\beta\!s_{ik}+\alpha h_k\ ^\beta\!s_{ij}=0.
\]
Contracting it with $b^jb^k$ yields
\be\label{P5}
(b^2-s^2)\Big[(-3s+b^2-s^2)\ ^\beta\!s_{i0}-2\alpha\ ^\beta\!s_{i}    \Big]=0.
\ee
Contracting \eqref{P5} with $b^i$ leads to
\[
(-3s+b^2-s^2)\ ^\beta\!s_{0}=0.
\]
If $-3s+b^2-s^2=0$, then $-3\alpha\beta+b^2\alpha^2-\beta^2=0$. Separating it in rational and irrational terms of $y^i$,  we get $\beta=0$. But this leads to a contradiction. Then  $^\beta\!s_{0}=0$, that is $^\beta\!s_{i}=0$. Putting $^\beta\!s_{i}=0$ in \eqref{P5} yields $^\beta\!s_{i0}=0$.
Substituting it into \eqref{P3} implies that $^\gamma\!s_{i0}=0$. From $^\beta\!s_{i0}=0$ and $^\gamma\!s_{i0}=0$, we get
\[
^\beta\!s_{ij}=0, \qquad ^\gamma\!s_{ij}=0.
\]
{\bf Case 2:} Let $^\beta\!r_{00}=k(x)\alpha^2$. Substituting $^\beta\!r_{00}=k(x)\alpha^2$ into \eqref{P1} implies that
\be\label{P6}
(\alpha^2+b^2\alpha^2-\beta^2)(e^{\frac{\beta}{\alpha}}\ ^\beta\!s_0+^\gamma\!\bar{s}_0)+\beta(b^2\alpha^2-\beta^2)e^{\frac{\beta}{\alpha}}k(x)=0.
\ee
 The term of \eqref{P6} which does not contain $\alpha^2$ is $-\beta^2(e^{\frac{\beta}{\alpha}}\ ^\beta\!s_0+^\gamma\!\bar{s}_0)-\beta^3e^{\frac{\beta}{\alpha}}k(x)$. Then we have
\[
(e^{\frac{\beta}{\alpha}}\ ^\beta\!s_0+^\gamma\!\bar{s}_0)=-\beta e^{\frac{\beta}{\alpha}}k(x).
\]
Putting it into \eqref{P6} yields $-\alpha^2\beta e^{\frac{\beta}{\alpha}} k(x)=0$. This implies that $k(x)=0$, then $^\beta\!r_{00}=0$. Therefore similar to case 1, we can conclude that $^\beta\!s_{ij}= ^\gamma\!s_{ij}=0$.

\end{proof}

\section{Douglas Spaces by  $(\alpha, \beta, \gamma)$-metrics}

In \cite{JD}, Douglas introduced the local functions $D^i_{j\;kl}$ on $TM_0$ defined by
\[
D^i_{j\:kl}:=\frac{\partial^3}{\partial y^j\partial y^k\partial y^l}\Big( G^i-\frac{1}{n+1}\frac{\partial G^m}{\partial y^m}y^i \Big).
\]
It is easy to verify that $D:=D^i_{j\:kl}dx^j\otimes\frac{\partial}{\partial x^i}\otimes dx^k\otimes dx^l$ is a well-defined tensor
on $TM_0$. $D$ is called the Douglas tensor. The Finsler space $(M,F)$ is called a Douglas space if and only if $G^iy^j-G^jy^i$ is homogeneous polynomial of degree three in $y^i$ \cite{BM}.\\
By \eqref{G^i} one can gets
\[
G^iy^j-G^jy^i=(^\alpha G^iy^j-^\alpha\!G^jy^i)+B^{ij},
\]
where
\beq\nonumber\label{B^ij}
B^{ij}:=\!\!\!\!\!\!\!\!\!\!\!&&\frac{\alpha}{\Pi}\big[ \Psi_s\ (^\beta\!s^i_0y^j-^\beta\!s^j_0y^i)+\Psi_{\bar{s}}\ (^\gamma\!s^i_0y^j- ^\gamma\!s^j_0y^i) \big]\\\nonumber
&&+\frac{1}{2\Gamma}
\Big\{ \big[\Psi_{ss}+(g^2-\bar{s}^2)J \big]\mathcal{R}^\beta+\big[\Psi_{s\bar{s}}-(\theta-s\bar{s})J \big]\ \mathcal{R}^\gamma \Big\}
(b^iy^j-b^jy^i)\\\nonumber
&&+\frac{1}{2\Gamma}
\Big\{ \big[\Psi_{\bar{s}\bar{s}}+(b^2-s^2)J \big]\mathcal{R}^\gamma+\big[\Psi_{s\bar{s}}-(\theta-s\bar{s})J \big]\mathcal{R}^\beta \Big\}(\gamma^iy^j-\gamma^jy^i).
\\
\eeq

\begin{ex}
Let $F$ be the metric that introduced in Example \ref{Ex1}. We can prove  $(\alpha, \beta, \gamma)$-metric $F =\alpha e^{\frac{\beta}{\alpha}}+\gamma$  is Daglus  if and only if $\beta$ is parallel with respect to $\alpha$ and $\gamma$ is closed.
\end{ex}
\begin{proof}
Substituting \eqref{AG} into \eqref{B^ij} implies that
\beqn
B^{ij}=\!\!\!\!\!\!\!\!&&\frac{\alpha^2}{(\alpha-\beta)e^\frac{\beta}{\alpha}}\Big[ (^\beta\!s^i_0y^j-^\beta\!\!s^j_0y^i)e^{\frac{\beta}{\alpha}}+(^\gamma\!s^i_0y^j-^\gamma\!\!s^j_0y^i) \Big]\\
&&+\frac{\alpha^2}{2\big[\alpha^2-\alpha\beta+b^2\alpha^2-\beta^2 \big]}\Big[\ ^\beta\!r_{00}-\frac{2\alpha^2}{(\alpha-\beta)e^{\frac{\beta}{\alpha}}} (e^{\frac{\beta}{\alpha}}\ ^\beta\!s_0+^\gamma\!\!\bar{s}_0)\Big](b^iy^j-b^jy^i).
\eeqn
Suppose that $F$ is a Douglas space, that is $B^{ij}$ are $hp(3)$.
Multiplying this equation by $2(\alpha-\beta)\big[\alpha^2-\alpha\beta+b^2\alpha^2-\beta^2 \big]e^{\frac{\beta}{\alpha}}$ yields
\beqn
&&2(\alpha-\beta)\big[\alpha^2-\alpha\beta+b^2\alpha^2-\beta^2 \big]e^{\frac{\beta}{\alpha}}B^{ij}=\\
&&2\alpha^2\big[\alpha^2-\alpha\beta+b^2\alpha^2-\beta^2 \big]e^{\frac{\beta}{\alpha}}(^\beta\!s^i_0y^j-^\beta\!\!s^j_0y^i)
+2\alpha^2\big[\alpha^2-\alpha\beta+b^2\alpha^2-\beta^2 \big](^\gamma\!s^i_0y^j-^\gamma\!\!s^j_0y^i)\\
&&
+\Big[ \alpha^2(\alpha-\beta)e^{\frac{\beta}{\alpha}}\ ^\beta\!r_{00}-2\alpha^4e^{\frac{\beta}{\alpha}}\ ^\beta\!s_0-2\alpha^4\ ^\gamma\bar{s}_0
 \Big](b^iy^j-b^jy^i).
\eeqn
Separating it in rational and irrational terms of $y^i$,  we obtain two equations as follows:
\beq\nonumber\label{irrational}
&&2(\alpha^2+b^2\alpha^2)e^{\frac{\beta}{\alpha}}B^{ij}=-2\alpha^2\beta e^{\frac{\beta}{\alpha}}(^\beta\!s^i_0y^j-^\beta\!\!s^j_0y^i)
-2\alpha^2\beta(^\gamma\!s^i_0y^j-^\gamma\!\!s^j_0y^i)+\alpha^2 e^{\frac{\beta}{\alpha}}\ ^\beta\!r_{00}(b^iy^j-b^jy^i).\\
\eeq
and
\beq\nonumber
&&-2\beta(2\alpha^2+b^2\alpha^2-\beta^2)e^{\frac{\beta}{\alpha}}B^{ij}=\\ \nonumber
&&2\alpha^2(\alpha^2+b^2\alpha^2-\beta^2) e^{\frac{\beta}{\alpha}}(^\beta\!s^i_0y^j-^\beta\!\!s^j_0y^i)
+2\alpha^2(\alpha^2+b^2\alpha^2-\beta^2)(^\gamma\!s^i_0y^j-^\gamma\!\!s^j_0y^i)\\ \nonumber
&&+
\Big[ -\alpha^2\beta e^{\frac{\beta}{\alpha}}\ ^\beta\!r_{00}-2\alpha^4e^{\frac{\beta}{\alpha}}\ ^\beta\!s_0-2\alpha^4\ ^\gamma\bar{s}_0 \Big]
(b^iy^j-b^jy^i).\\ \label{rational}
\eeq
Eliminating $B^{ij}$ from \eqref{irrational} and \eqref{rational} yields
\beqn
&&(\alpha^2-b^2\alpha^2)\Big\{ 2\alpha^2(\alpha^2+b^2\alpha^2-\beta^2) e^{\frac{\beta}{\alpha}}(^\beta\!s^i_0y^j-^\beta\!\!s^j_0y^i)
+2\alpha^2(\alpha^2+b^2\alpha^2-\beta^2)(^\gamma\!s^i_0y^j-^\gamma\!\!s^j_0y^i)\\
&&+
\Big[ -\alpha^2\beta e^{\frac{\beta}{\alpha}}\ ^\beta\!r_{00}-2\alpha^4e^{\frac{\beta}{\alpha}}\ ^\beta\!s_0-2\alpha^4\ ^\gamma\bar{s}_0 \Big]
(b^iy^j-b^jy^i) \Big\}=\\
&&-\beta(2\alpha^2-b^2\alpha^2-\beta^2)\Big\{ -2\alpha^2\beta e^{\frac{\beta}{\alpha}}(^\beta\!s^i_0y^j-^\beta\!\!s^j_0y^i)
-2\alpha^2\beta(^\gamma\!s^i_0y^j-^\gamma\!\!s^j_0y^i)+\alpha^2 e^{\frac{\beta}{\alpha}}\ ^\beta\!r_{00}(b^iy^j-b^jy^i) \Big\}.
\eeqn
Simplifying this equation implies that
\beq\nonumber\label{N0}
&&2\Big[ (\alpha^2+b^2\alpha^2-\beta^2)^2-\alpha^2\beta^2 \Big]e^{\frac{\beta}{\alpha}}(^\beta\!s^i_0y^j-^\beta\!\!s^j_0y^i)+
2\Big[ (\alpha^2+b^2\alpha^2-\beta^2)^2-\alpha^2\beta^2 \Big](^\gamma\!s^i_0y^j-^\gamma\!\!s^j_0y^i)\\ \nonumber
&&+\Big[ -(\alpha^2+b^2\alpha^2)\big( \beta e^{\frac{\beta}{\alpha}}\ ^\beta\!r_{00}+2\alpha^2e^{\frac{\beta}{\alpha}}\ ^\beta\!s_0+2\alpha^2\ ^\gamma\bar{s}_0  \big)+\beta e^{\frac{\beta}{\alpha}}\ ^\beta\!r_{00}(2\alpha^2+b^2\alpha^2-\beta^2) \Big](b^iy^j-b^jy^i)=0.\\
\eeq
Contracting it with $b_iy_j$, we get
\be\label{N1}
2\alpha^2(\alpha^2-\beta^2)(\alpha^2+b^2\alpha^2-\beta^2)(e^{\frac{\beta}{\alpha}}\ ^\beta\!s_0+^\gamma\!\bar{s}_0)+\beta(\alpha^2-\beta^2)(b^2\alpha^2-\beta^2)e^{\frac{\beta}{\alpha}}\ ^\beta\!r_{00}=0.
\ee
The term of \eqref{N1} which does not contain $\alpha^2$ is $-\beta^3e^{\frac{\beta}{\alpha}}\ ^\beta\!r_{00}$. Notice that $-\beta^3e^{\frac{\beta}{\alpha}}$ is not divisible by $\alpha^2$, then $^\beta\!r_{00}=k(x)\alpha^2$ and we can  consider two cases:
\ben
\item [(1)]  $k(x)=0$,
\item [(2)]  $k(x)\neq0$.
\een
{\bf Case 1:} Substituting $^\beta\!r_{00}=0$ into \eqref{N1} implies that
\[
2\alpha^2(\alpha^2+b^2\alpha^2-\beta^2)(e^{\frac{\beta}{\alpha}}\ ^\beta\!s_0+^\gamma\!\bar{s}_0)=0
\]
If $\alpha^2+b^2\alpha^2-\beta^2=0$, then the term which does not contain $\alpha^2$  is $\beta^2$. This implies that $\beta^2=0$ which leads to a contradiction. Hence
\be\label{N2}
e^{\frac{\beta}{\alpha}}\ ^\beta\!s_0+^\gamma\!\bar{s}_0=0.
\ee
Putting $^\beta\!r_{00}=0$ and \eqref{N2} into \eqref{N0} leads to
\[
\Big[ (\alpha^2+b^2\alpha^2-\beta^2)^2-\alpha^2\beta^2 \Big]\Big[e^{\frac{\beta}{\alpha}}(^\beta\!s^i_0y^j-^\beta\!\!s^j_0y^i)+
(^\gamma\!s^i_0y^j-^\gamma\!\!s^j_0y^i)\Big]=0.
\]
 By a similar argument, we get  $(\alpha^2+b^2\alpha^2-\beta^2)^2-\alpha^2\beta^2 \neq0$. Therefore
\be\label{N3}
e^{\frac{\beta}{\alpha}}(^\beta\!s^i_0y^j-^\beta\!\!s^j_0y^i)+
(^\gamma\!s^i_0y^j-^\gamma\!\!s^j_0y^i)=0.
\ee
Contracting \eqref{N3} with $y_j$ yields
\be\label{N4}
e^{\frac{\beta}{\alpha}}\ ^\beta\!s^i_0+ ^\gamma\!s^i_0=0 \Longrightarrow e^{\frac{\beta}{\alpha}}\ ^\beta\!s_{i0}+ ^\gamma\!s_{i0}=0.
\ee
Differentiating \eqref{N4} with respect to  $y^j$ and  $y^k$ and multiplying it by $\alpha^2$ imply that
\[
-(\alpha_jh_k+\alpha_kh_j-s\alpha\alpha_{jk})\ ^\beta\!s_{i0}+h_jh_k\ ^\beta\!s_{i0}+\alpha h_j\ ^\beta\!s_{ik}+\alpha h_k\ ^\beta\!s_{ij}=0.
\]
Contracting it with $b^jb^k$ yields
\be\label{N5}
(b^2-s^2)\Big[(-3s+b^2-s^2)\ ^\beta\!s_{i0}-2\alpha\ ^\beta\!s_{i}    \Big]=0.
\ee
Contracting \eqref{N5} with $b^i$ leads to
\[
(-3s+b^2-s^2)\ ^\beta\!s_{0}=0.
\]
If $-3s+b^2-s^2=0$, then $-3\alpha\beta+b^2\alpha^2-\beta^2=0$. Separating it in rational and irrational terms of $y^i$,  we get $\beta=0$. But this leads to a contradiction. Then  $^\beta\!s_{0}=0$, that is $^\beta\!s_{i}=0$. Putting $^\beta\!s_{i}=0$ in \eqref{N5} yields $^\beta\!s_{i0}=0$.
Substituting it into \eqref{N4} implies that $^\gamma\!s_{i0}=0$. From $^\beta\!s_{i0}=0$ and $^\gamma\!s_{i0}=0$, we get
\[
^\beta\!s_{ij}=0, \qquad ^\gamma\!s_{ij}=0.
\]
{\bf Case 2:} Let $^\beta\!r_{00}=k(x)\alpha^2$. Putting $^\beta\!r_{00}=k(x)\alpha^2$ into \eqref{N1} implies that
\be\label{N6}
2(\alpha^2+b^2\alpha^2-\beta^2)(e^{\frac{\beta}{\alpha}}\ ^\beta\!s_0+^\gamma\!\bar{s}_0)+\beta(b^2\alpha^2-\beta^2)e^{\frac{\beta}{\alpha}}k(x)=0.
\ee
 The term of \eqref{N6} which does not contain $\alpha^2$ is $-2\beta^2(e^{\frac{\beta}{\alpha}}\ ^\beta\!s_0+^\gamma\!\bar{s}_0)-\beta^3e^{\frac{\beta}{\alpha}}k(x)$. Then we have
\[
2(e^{\frac{\beta}{\alpha}}\ ^\beta\!s_0+^\gamma\!\bar{s}_0)=-\beta e^{\frac{\beta}{\alpha}}k(x).
\]
Putting it into \eqref{N6} yields $-\alpha^2\beta e^{\frac{\beta}{\alpha}} k(x)=0$. This implies that $k(x)=0$, then $^\beta\!r_{00}=0$. Therefore similar to case 1, we can conclude that $^\beta\!s_{ij}= ^\gamma\!s_{ij}=0$.

\end{proof}

\bigskip
\noindent
 Nasrin Sadeghzadeh and Tahere Rajabi \\
Department of Mathematics, Faculty of Science,\\
University of Qom, Qom, Iran.\\
E-mail: nsadeghzadeh@qom.ac.ir\\
E-mail: t.rajabi.j@gmail.com\\

\begin{thebibliography}{MaHo}
\bibitem{BM} S. Bacso, M. Matsumoto, {\it On the Finsler spaces of Douglas type. A generalization of the notion of Berwald space}, Publ. Math. Debrecen, \textbf{51}(1997), 385-406.
\bibitem{Shen} S. S. Chern, Z. Shen, {\it Riemann-Finsler Geometry}. World Scientific Publishing Co. Pte. Ltd.,
Hackensack, NJ, (2005).
\bibitem{JD} J. Douglas, {\it The general geometry of paths}, Ann. Math. \textbf{29} (1927-1928), 143-168.
\bibitem{H} G. Hamel, {\it $\ddot{U}$ber die Geometrieen in denen die Geraden die K$\ddot{u}$rzesten sind}, Math. Ann. \textbf{57}(1903), 231-264.
\bibitem{M1} M. Matsumoto, {\it On Finsler spaces with Randers metric and special forms of important tensors}, J. Math. Kyoto Univ., \textbf{14}-3 (1974), 477-498.
\bibitem{RS} T. Rajabi, N. Sadeghzadeh, {\it A new class of Finsler metrics}, .
\bibitem{YZ} C. Yu, H. Zhu,{\it On a new class of Finsler metrics}, Differential Geom. Appl. \textbf{29}-2 (2011), 244-254.
\end{thebibliography}
\end{document}